\documentclass[12pt,reqno,a4paper]{amsart}
\usepackage{amssymb}
\usepackage{amsmath}
\usepackage{amsthm}
\usepackage{lineno}
\makeatletter
\@namedef{subjclassname@2010}{%
  \textup{2010} Mathematics Subject Classification}
\makeatother
\numberwithin{equation}{section}

\newtheorem{lemma}{Lemma}[section]
\newtheorem{theorem}{Theorem}[section]

\title[Applications of a new P-Q modular equation of degree two]
{Applications of a new P-Q modular equation of degree two}
\author[D. J. Prabhakaran]{D. J. Prabhakaran}
\address{Department of Mathematics \\
Anna University, MIT Campus\\
Chennai- 600025\\ India}
\email{asirprabha@gmail.com}
\author[K. Ranjith kumar]{K. Ranjith kumar}
\address{Department of Mathematics \\
Anna University, MIT Campus\\
Chennai- 600025\\ India}
\email{ranjithkrkkumar@gmail.com}
\begin{document}
\begin{abstract}
At scattered places in his first notebook, Ramanujan recorded the values for 107 class invariants or irreducible monic polynomials satisfied by them. On pages 294-299 in his second notebook, he gave a table of values for 77 class invariants $G_n$ and $g_n$ in his second notebook. Traditionally, $G_n$ is determined for odd values of $n$ and $g_n$ for even values of $n$. On pages 338 and 339 in his first notebook, Ramanujan defined the remarkable product of theta-functions $a_{m, n}$. Also, he recorded eighteen explicit values depending on two parameters, namely, $m$, and $n$, where these are odd integers. In this paper, we initiate to study explicit evaluations of $G_n$ for even values of $n$. We establish a new general formula for the explicit evaluations of $G_n$ involving class invariant $g_n$. For this purpose, we derive a new P-Q modular equation of degree two. Further application of this modular equation, we establish a new formula to explicit evaluation of $a_{m, 2}$. Also, we compute several explicit values of class invariant $g_{n}$ and singular moduli $\alpha_n$.
\end{abstract}
\keywords{Modular equation, theta functions, class invariants.}
\maketitle
\section{Introduction}
The following definitions of theta functions \cite{Berndt-notebook-3} $\varphi$, $\psi,$ and $f$ with $|q|<1$ are classical:
\begin{eqnarray*}
\varphi(q) &=& f(q,q)=\sum^{\infty}_{n=-\infty}q^{n^2}=(-q;q^2)^2_{\infty}(q^2;q^2)_{\infty},\\
\psi(q) &=& f(q,q^3)=\sum^{\infty}_{n=0}q^{\frac{n(n+1)}{2}}=\frac{(q^2;q^2)_{\infty}}{(q;q^2)^2_{\infty}},\\
f(-q)&=&f(-q,-q^2)=\sum^{\infty}_{n=-\infty}(-1)^nq^{\frac{n(3n-1)}{2}}=(q;q)_{\infty},
\end{eqnarray*}
where, $\displaystyle (a;q)_{\infty}=\prod^{\infty}_{n=0}\left(1-aq^n\right)$.\\
For $q=e^{-\pi\sqrt{n}}$, Weber-Ramanujan class invariants \cite[p.183, (1.3)]{Berndt-notebook-5} are defined by
\begin{align}\label{ejhq1}
G_n = 2^{-1/4}q^{-1/24} \chi(q)
\quad ; \quad
g_n = 2^{-1/4}q^{-1/24} \chi(-q),
\end{align}
where, $n$ is a positive rational number and $\chi(q)=(-q;q^2)_\infty$. Ramanujan \cite[Entry 2.1, p.187]{Berndt-notebook-5} recorded simple formula relating these class invariants as follows:
\begin{eqnarray}\label{identy}
g_{4n} &=& 2^{1/4}g_nG_n.
\end{eqnarray}
Ramanujan evaluated a total of 116 class invariants \cite[p.189-204]{Berndt-notebook-5}. Traditionally, $G_n$ is determined for odd values of $n$ and $g_n$ for even values of $n$.  These have been proved by various authors using techniques such as modular equations, Kronecker limit formula, and empirical process (established by Watson) \cite[Chapter 34]{Berndt-notebook-5}.\\
The ordinary or Gaussian hypergeometric function is defined by
\begin{eqnarray*}
{_2F_1}\left(a,b;c;z
\right)
&=&\sum_{n=0}^{\infty}\frac{(a)_n(b)_n}{(c)_n n!}z^n, \ \ \ \ |z|< 1
\end{eqnarray*}
where, $a, b, c$ are complex numbers such that $ c \neq 0, -1, -2, \ldots,$ and $(a)_0 = 1, \ (a)_n = a(a+1)(a+2)\ldots(a+n-1)$ for any positive integer $n$.

Now, we shall recall the definition of modular equation from \cite{Berndt-notebook-3}. The complete elliptic integral of the first kind $K(k)$ of modulus $k$ is defined by
\begin{align}\label{112235eq0}
K(k)=\int_{0}^{\frac{\pi}{2}}\frac{d\theta}{\sqrt{1-k^2\sin^2 \theta}}=\frac{\pi}{2}\sum^{\infty}_{n=0}\frac{\left(\frac{1}{2}\right)^2_n}{(n!)^2}k^{2n}=
\frac{\pi}{2}\varphi^2\left(e^{-\pi\frac{K'}{K}}\right),
\quad (0<k<1)
\end{align}
and let $K'=K(k'),$ where $k'=\sqrt{1-k^2}$ is represented as the complementary modulus of $k$. Let $K, K', L,$ and $L'$ denote the complete elliptic integrals of the first kind associated with the moduli $k, k', l,$ and $l'$ respectively. In case, the equality
\begin{equation} \label{eq00}
n\frac{K'}{K}=\frac{L'}{L}
\end{equation}
holds for a positive integer $n$, then a modular equation of degree $n$ is the relation between the moduli $k,$ and $l$, which is implied by equation (\ref{eq00}). Ramanujan defined his modular equation involving $\alpha,$ and $\beta$, where, $\alpha=k^2,$ and $\beta=l^2$. Then we say $\beta$ is of degree $n$ over $\alpha$.

Ramanujan recorded 23 $P$-$Q$ modular equations in terms of their theta function in his notebooks \cite{sr}. All those proved by Berndt et al. by employing the theory of theta functions and modular forms.

If, as usually quoted in the theory of elliptic functions, $k = k(q)$ denotes the modulus, then, the singular moduli $k_n$ is defined by $k_n = (e^{-\pi\sqrt{n}})$, where $n$ is a positive integer. In terms of Ramanujan, set $\alpha = k^2$ and $\alpha_n = k^2_n,$ he hypothesized the values of over 30 singular moduli in his notebooks. On page 82 of his first notebook, Ramanujan stated three additional theorems for calculating $\alpha_n$ for even values of $n$. Particularly, he offered formulae for $\alpha_{4p}$, $\alpha_{8p},$ and $\alpha_{16p}$. Moreover, he recorded several values of $\alpha_n$ for odd values of $n$ in his first and second notebook. All these results have proved by Berndt et al. by employing  Ramanujan's class invariants $G_n$ and $g_n$. Also we observed that representation for $\alpha_n$ in terms of theta function. This is given by
\cite[Entry 12 (i),(iv) Ch.17, p.124]{Berndt-notebook-3}
\begin{eqnarray}\label{e87q1}
\alpha_n &=& \left(\frac{f(q)}{2^{1/2}q^{1/8}f(-q^4)}\right)^{-8},
\end{eqnarray}
where, $\displaystyle q = e^{-\pi\sqrt{n}}.$

On page 338 in his first notebook \cite{sr}, Ramanujan defined
\begin{equation} \label{eq1}
 a_{m,n} =\frac{nq^{(n-1)/4}\psi^2\left(q^n\right)\varphi^2\left(-q^{2n}\right)}{\psi^2\left(q\right)\varphi^2\left(-q^2\right)}
\end{equation}
where, $\displaystyle q=e^{-\pi\sqrt{m/n}}$ and $m$, $n$ are positive rationals then, on page 338 and 339, he offered a list of 18 particular values. All those 18 values have proved by Berndt, Chan, and Zhang \cite{be5}. Recently, Prabhakaran, and Ranjith Kumar \cite{djpkrk} have established a new general formulae for the explicit evaluations of $a_{3m,3},$ and $a_{m,9}$ by using $P-Q$ mixed modular equations, and the values for certain class invariant of Ramanujan. Also they have calculated some new explicit values of $a_{3m,3}$ for $m = 2, 7, 13, 17, 25, 37,$ and $a_{m,9}$ for $m = 17, 37$.

Naika, and Dharmendra \cite{ms2} have given alternative form of \eqref{eq1} as follows:
\begin{equation} \label{dseq2}
a_{m,n}=\frac{nq^{(n-1)/4}\psi^2\left(-q^n\right)\varphi^2\left(q^{n}\right)}{\psi^2\left(-q\right)\varphi^2\left(q\right)}.
\end{equation}
They have proved some general theorems to calculate explicit values of $a_{m,n}$.

The organisation of the present study is as follows. In Section 2, we collect some identities which are useful in proofs of our main results. In Section 3, we derive a new $P-Q$ modular equation of degree two. Applying this modular equation, we establish new general formulae for
the explicit evaluations of class invariant $G_n$ for even values of $n,$ and the Ramanujan's remarkable product of theta functions $a_{m, 2}$ along with class invariant $g_n$. By using these formulae, we compute several explicit values of class invariant $G_n,$ and $a_{m, 2}$. Also, we evaluate several explicit values of class invariant $g_{n},$ and singular moduli $\alpha_n$. These are presented in Section 4.
\section{Preliminaries}
We list a few identities which are useful in establishing our main results.
\begin{lemma}\label{lem1}\cite[Entry 24(iii)  p. 39]{Berndt-notebook-3} We have
\begin{equation} \label{1lem}
  f(q)f(-q^2)= \psi(-q)\varphi(q).
\end{equation}
\end{lemma}
\begin{lemma}\cite[Entry 12(i),(iii) Ch.17, p.124]{Berndt-notebook-3} We have
\begin{eqnarray}
f(q)&=&\sqrt{z} 2^{-1/6}\left(\alpha(1-\alpha)/q\right)^{1/24} \label{feq1},\\
f(-q^2)&=&\sqrt{z} 2^{-1/3}\left(\alpha(1-\alpha)/q\right)^{1/12}\label{feq2}.
\end{eqnarray}
\end{lemma}
\begin{lemma}\cite[Entry 24(ii), p. 214]{Berndt-notebook-3} If $\beta$ is of degree 2 over $\alpha$, then,
\begin{eqnarray}\label{e22qn1}
m\sqrt{1-\alpha}+\sqrt{\beta} &=& 1, \\ \label{e2q2n1}
m^2\sqrt{1-\alpha}+\beta &=& 1,
\end{eqnarray}
where, $m = z_1/z_2 = \varphi^2(q)/\varphi^2(q^2)$.
\end{lemma}
\begin{lemma}\cite[Theorem 3.5.1]{Yi-Thesis-2000} \label{thkhe2}
If  $P=\displaystyle \frac{f(-q)}{q^{1/8}f(-q^{4})}$  and  $\displaystyle Q = \frac{f(-q^2)}{q^{1/4}f(-q^{8})}$, then,
\begin{eqnarray}
\left(PQ\right)^4+\left(\frac{4}{PQ}\right)^4&=& \left(\frac{Q}{P}\right)^{12}-16\left(\frac{P}{Q}\right)^4-16\left(\frac{Q}{P}\right)^4.\label{ttthe2}
\end{eqnarray}
\end{lemma}
\section{General formulae for the explicit evaluations of $G_{2n}$, $G_{n/2}$, and $a_{m, 2}$}
In this section, we derive a new P-Q modular equation of degree two. As application of this modular equation, we establish some general formulae for the explicit evaluations of $G_{2n}$, $G_{n/2}$, and $a_{m, 2}$ in term of the class invariant $g_n$.
\begin{theorem}
If  $P=\displaystyle \frac{f(q)}{q^{1/24}f(q^{2})}$  and  $\displaystyle Q = \frac{f(-q^2)}{q^{1/12}f(-q^{4})}$, then,
\begin{eqnarray}\label{the1}
Q^{16}-P^4Q^{14}+8P^4Q^2-4P^8&=&0.
\end{eqnarray}
\end{theorem}
\begin{proof}
Transcribing $P,$ and $Q$ using  \eqref{feq1}, and \eqref{feq2}, then simplifying, we arrive at
\begin{align}\label{eqn1}
P=\sqrt{\frac{z_1}{z_2}}\left(\frac{\alpha(1-\alpha)}{\beta(1-\beta)}\right)^{1/24}
\quad ; \quad
Q=\sqrt{\frac{z_1}{z_2}}\left(\frac{\alpha(1-\alpha)}{\beta(1-\beta)}\right)^{1/12},
\end{align}
where, $\beta$ is of degree 2 over $\alpha$. It follow that
\begin{align}\label{e2qn1}
\frac{Q}{P}=\left(\frac{\alpha(1-\alpha)}{\beta(1-\beta)}\right)^{1/24}
\quad ; \quad
m =\frac{P^4}{Q^2}.
\end{align}
Now isolating $\alpha,$ and $\beta$ from \eqref{e22qn1}, and \eqref{e2q2n1}, we deduce that
\begin{align}\label{e2qnn1}
\alpha = \frac{4(m-1)}{m^2}
\quad ; \quad
\beta = \left(m-1\right)^2.
\end{align}
By \eqref{e2qnn1}, we observe that
\begin{align}\label{en2qnn1}
\alpha(1-\alpha) = \frac{4(m-1)(m-2)^2}{m^4}
\quad  ; \quad
\beta(1-\beta) = -m(m-1)^2(m-2).
\end{align}
Employing \eqref{en2qnn1} in first term of \eqref{e2qn1}, and simplifying, we arrive at
\begin{eqnarray*}
\left(m-1\right)\left(m-2\right)\left(m^6Q^{24}-m^5Q^{24}+4mP^{24}-8P^{24}\right)&=&0.
\end{eqnarray*}
We observe that the last factor of the above equation vanish for $q\rightarrow 0,$ whereas, the other factors does not vanish for that specific value. Thus, we obtain that
\begin{eqnarray*}\label{5e9n}
m^6Q^{24}-m^5Q^{24}+4mP^{24}-8P^{24}&=& 0.
\end{eqnarray*}
Now applying the value of $m$ in the above equation, we complete the proof.
\end{proof}
\begin{theorem}\label{thfe3}
If $n$ is any positive rational, and
\begin{eqnarray*}
\Lambda &=& \frac{g^{12}_{2n}+g^{-12}_{2n}}{2},
\end{eqnarray*}
then,
\begin{eqnarray}\label{t3}
\frac{G_{2n}}{G_{n/2}} &=& \frac{1}{g_{2n}}\left(\sqrt{\Lambda}+\sqrt{\Lambda-1}\right)^{1/4},
\end{eqnarray}
$\displaystyle  \left(\sqrt{2}G_{2n}G_{n/2}\right)^{12}-16\left(\left(\sqrt{2}G_{2n}G_{n/2}\right)^4+\left(\sqrt{2}G_{2n}G_{n/2}\right)^{-4}\right)$
\begin{eqnarray}\label{thev3}
&=& 16g^{12}_{2n}\left(\sqrt{\Lambda}+\sqrt{\Lambda-1}\right)+
 16g^{-12}_{2n}\left(\sqrt{\Lambda}-\sqrt{\Lambda-1}\right).
\end{eqnarray}
\end{theorem}
\begin{proof}
Solving \eqref{the1} for $P/Q,$ and choosing the appropriate root, we obtain that
\begin{eqnarray}\label{tko3}
\frac{P}{Q} &=& \frac{Q}{2^{3/4}}\left(\sqrt{Q^{12}+\frac{64}{Q^{12}}}-\sqrt{Q^{12}+\frac{64}{Q^{12}}-16}\right)^{1/4}.
\end{eqnarray}
We observed that some representations for $G_n$ and $g_n$  in terms of $f(q)$ and $f(-q)$ by Entry 24(iii) \cite[p.39]{Berndt-notebook-3} as follow:
\begin{align}\label{eqgn1}
G_n = \frac{f(q)}{2^{1/4}q^{1/24}f(-q^2)}
\quad ; \quad
g_n = \frac{f(-q)}{2^{1/4}q^{1/24}f(-q^2)}.
\end{align}
Employing \eqref{eqgn1} in \eqref{tko3} along with $q=e^{-\pi\sqrt{n/2}}$, we arrive at \eqref{t3}. Now, setting $q=e^{-\pi\sqrt{n/2}}$ in Lemma \ref{thkhe2} and employing the definition of $g_n$, we obtain that
\begin{align}\label{trs}
PQ = 2g^2_{2n}g_{n/2}g_{8n}
\quad ; \quad
\frac{P}{Q} = \frac{g_{n/2}}{g_{8n}}.
\end{align}
Now applying \eqref{identy} in \eqref{trs}, we deduce that
\begin{align}\label{esqgn1}
PQ = \frac{2g^4_{2n}G_{2n}}{G_{n/2}}
\quad ; \quad
\frac{P}{Q} = \frac{1}{\sqrt{2}G_{2n}G_{n/2}}.
\end{align}
By \eqref{t3}, we observe that
\begin{eqnarray}\label{tgjhe3}
\frac{g^4_{2n}G_{2n}}{G_{n/2}} &=& g^3_{2n}\left(\sqrt{\frac{g^{12}_{2n}+g^{-12}_{2n}}{2}}+\sqrt{\frac{g^{12}_{2n}+g^{-12}_{2n}}{2}-1}\right)^{1/4}.
\end{eqnarray}
Now applying \eqref{esqgn1}, and \eqref{tgjhe3} in \eqref{ttthe2}, we obtain \eqref{thev3}.
\end{proof}
\begin{theorem}\label{the3}
If $m$ is any positive rational, then
\begin{eqnarray*}
a_{m, 2} &=& \frac{1}{g^{6}_{2m}}\left(\sqrt{\frac{g^{12}_{2m}+g^{-12}_{2m}}{2}}+\sqrt{\frac{g^{12}_{2m}+g^{-12}_{2m}}{2}-1}\right)^{1/2}.
\end{eqnarray*}
\end{theorem}
\begin{proof}
Solving \eqref{the1} for $P^4Q^4,$ and choosing the appropriate root, we arrive at
\begin{eqnarray}\label{tkjhhe3}
P^4Q^4 &=& \frac{Q^{12}}{8}\left(\sqrt{Q^{12}+\frac{64}{Q^{12}}}-\sqrt{Q^{12}+\frac{64}{Q^{12}}-16}\right).
\end{eqnarray}
Let $q=e^{-\pi\sqrt{m/2}}$, then the identity \eqref{dseq2}, becomes
\begin{equation} \label{weq2}
 a_{m,2}=\frac{2q^{1/4}\psi^2\left(-q^2\right)\varphi^2\left(q^{2}\right)}{\psi^2\left(-q\right)\varphi^2\left(q\right)}.
 \end{equation}
Now applying \eqref{1lem} in \eqref{weq2}, we conclude that
\begin{equation} \label{eq2}
 a_{m,2}=\frac{2q^{1/4}f^2(q^2)f^2(-q^4)}{f^2(q)f^2(-q^2)}.
 \end{equation}
Employing the second term of \eqref{eqgn1} in \eqref{tkjhhe3} along with $q=e^{-\pi\sqrt{m/2}}$, then it follow that reporting in \eqref{eq2}, we arrive at desired result.
\end{proof}
\section{Explicit evaluations}
In section, we compute several explicit evaluations of class invariant $G_n$ for even values of $n,$ and $a_{m, 2}$ by using Theorem \ref{thfe3}, and Theorem \ref{the3} respectively. After obtaining class invariant $G_n$, then we evaluate several explicit evaluations of class invariant $g_n,$ and singular moduli $\alpha_n$.
\begin{theorem}\label{tkhe4}
We have
\begin{eqnarray*}
G_{46}&=& \frac{1}{2^{1/8}}\left(78\sqrt{2}+23\sqrt{23}\right)^{1/16}\left(\frac{5+\sqrt{23}}{\sqrt{2}}\right)^{1/8}
\left(\sqrt{\frac{3\sqrt{2}+8}{4}}+\sqrt{\frac{3\sqrt{2}+4}{4}}\right)^{1/2}\\ && \times\left(\sqrt{\frac{6\sqrt{2}+11}{4}}-\sqrt{\frac{6\sqrt{2}+7}{4}}\right)^{1/4}.
\end{eqnarray*}
\end{theorem}
\begin{proof}
From the table in Chapter 34 of Ramanujan’s notebooks \cite[p.201]{Berndt-notebook-5}, we have
\begin{eqnarray*}
g_{46} &=& \sqrt{\frac{3+\sqrt{2}+\sqrt{7+6\sqrt{2}}}{2}}.\label{1thec003}
\end{eqnarray*}
It follows that
\begin{eqnarray}
g^{12}_{46}+g^{-12}_{46}&=& 2646+1872\sqrt{2}.\label{1thec003}
\end{eqnarray}
Employing \eqref{1thec003} in \eqref{t3} with $n=23$, we conclude that
\begin{eqnarray}\nonumber
\frac{G_{46}}{G_{23/2}} &=& \left(2645+1872\sqrt{2}+\sqrt{14004792+9902880\sqrt{2}}\right)^{1/8}\\ && \times
 \left(\sqrt{\frac{6\sqrt{2}+11}{4}}-\sqrt{\frac{6\sqrt{2}+7}{4}}\right)^{1/2} .\label{1bthec003}
\end{eqnarray}
By $(9.5)$ \cite[ p.284]{Berndt-notebook-3}, observe that
\begin{eqnarray}
\sqrt{14004792+9902880\sqrt{2}} &=& 552\sqrt{23}+390\sqrt{46}.\label{2themmco003}
\end{eqnarray}
Reporting \eqref{2themmco003} in \eqref{1bthec003}, and further simplification, we obtain that
\begin{eqnarray}
\frac{G_{46}}{G_{23/2}} &=& \left(78\sqrt{2}+23\sqrt{23}\right)^{1/8}\left(\frac{5+\sqrt{23}}{\sqrt{2}}\right)^{1/4}
 \left(\sqrt{\frac{6\sqrt{2}+11}{4}}-\sqrt{\frac{6\sqrt{2}+7}{4}}\right)^{1/2} .\label{1thnvnm03}
\end{eqnarray}
Applying \eqref{1thec003} in \eqref{thev3}, and after a straightforward, lengthy calculation, we deduce that
\begin{eqnarray*}\nonumber
h^{32}-32h^{24}-\left(4356352+3080448\sqrt{2}\right)h^{20}+224h^{16}+\left(69701632+49287168\sqrt{2}\right)h^{12} && \\
-\left(3587934720+2537054208\sqrt{2}\right)h^8+\left(69701632+49287168\sqrt{2}\right)h^4+256 &=& 0 ,\label{1thnec003}
\end{eqnarray*}
where, $\displaystyle h = \sqrt{2}G_{46}G_{23/2}$. Now isolating the terms involving $\sqrt{2}$ on one side of the above equation, squaring both sides, and simplifying, we deduce that
\begin{eqnarray*}
\left(h^{16}-208h^{12}+456h^8-832h^4+16\right) \left(h^{48}+208h^{44}+42744h^{40}+84032h^{36}-1838096h^{32}\right. && \\ \left.
+33126912h^{28}+104902912h^{24}-1089853440h^{20}-761938176h^{16}+10947629056h^{12} \right.&& \\ \left. -7091652608h^{8}+2230665216h^{4}+4096\right)&=& 0.\label{mhnec003}
\end{eqnarray*}
A numerical calculation show that $h$ is not a root of the second factor. Since the first factor has positive roots, and it follows that
\begin{eqnarray*}
h^{16}-208h^{12}+456h^8-832h^4+16 &=& 0, \label{mhnec003}
\end{eqnarray*}
or equivalently,
\begin{eqnarray*}
\left(h^4+\frac{4}{h^4}\right)^2-208\left(h^4+\frac{4}{h^4}\right)+ 448 &=& 0.\label{vmhnec003}
\end{eqnarray*}
Solving the above quadric equation, and choosing the convenient root, we deduce that
\begin{eqnarray}
h^4+\frac{4}{h^4}&=& 104+72\sqrt{2}.\label{vmnnec003}
\end{eqnarray}
Again solving \eqref{vmnnec003} for $h,$ and $h > 1$, we obtain that
\begin{eqnarray}
G_{46}G_{23/2}&=& \frac{1}{2^{1/4}}\left(26+18\sqrt{2}+\sqrt{1323+936\sqrt{2}}\right)^{1/4}.\label{nec003}
\end{eqnarray}
Now we apply Lemma 9.10 \cite[ p.292]{Berndt-notebook-5} with $r=26+18\sqrt{2}$. Then $t=\left(3\sqrt{2}+6\right)/4,$ and so
\begin{eqnarray}\label{6th288880}
26+18\sqrt{2}+\sqrt{1323+936\sqrt{2}}&=&\left(\sqrt{\frac{3\sqrt{2}+8}{4}}+\sqrt{\frac{3\sqrt{2}+4}{4}}\right)^4.
\end{eqnarray}
Now combining \eqref{1thnvnm03}, \eqref{nec003}, and \eqref{6th288880}, we arrive at the desired result.
\end{proof}
\begin{theorem}\label{the4}
We have
\begin{eqnarray*}
G_{14}&=&\frac{1}{2^{1/8}}\left(2\sqrt{2}+\sqrt{7}\right)^{1/16}\left(\frac{3+\sqrt{7}}{\sqrt{2}}\right)^{1/8}
\left(\sqrt{\frac{2\sqrt{2}+3}{2}}+\sqrt{\frac{2\sqrt{2}+1}{2}}\right)^{1/4}\\ && \times\left(\sqrt{\frac{2\sqrt{2}+3}{4}}-\sqrt{\frac{2\sqrt{2}-1}{4}}\right)^{1/4},\\
G_{22}&=& \frac{1}{2^{1/8}}\left(\sqrt{2}-1\right)^{1/4}\left(3\sqrt{11}+7\sqrt{2}\right)^{1/8}\left(\frac{\sqrt{11}+3}{\sqrt{2}}\right)^{1/4},
\end{eqnarray*}
\begin{eqnarray*}
G_{34}&=&\frac{1}{2^{1/8}}\left(\sqrt{2}+1\right)^{1/4}\left(3\sqrt{2}+\sqrt{17}\right)^{1/8}
\left(\sqrt{\frac{\sqrt{17}+5}{4}}+\sqrt{\frac{\sqrt{17}+1}{4}}\right)^{1/2}\\ && \times\left(\sqrt{\frac{3\sqrt{17}+13}{8}}-\sqrt{\frac{3\sqrt{17}+5}{8}}\right)^{1/4},\\
G_{58}&=& \frac{1}{2^{1/8}}\left(\sqrt{2}+1\right)^{3/4}\left(13\sqrt{58}+99\right)^{1/8}\left(\frac{\sqrt{29}-5}{2}\right)^{1/4},\\
G_{70}&=& \frac{1}{2^{1/8}}\left(\sqrt{2}-1\right)^{1/4}\left(2\sqrt{2}+\sqrt{7}\right)^{1/4}\left(3\sqrt{14}+5\sqrt{5}\right)^{1/8} \left(\frac{\sqrt{7}+\sqrt{5}}{\sqrt{2}}\right)^{1/4}\\&&
\times\left(\frac{3+\sqrt{7}}{\sqrt{2}}\right)^{1/4}\left(\frac{\sqrt{5}-1}{2}\right)^{1/2},\\
G_{82}&=& \frac{1}{2^{1/8}}\left(\sqrt{2}+1\right)^{1/2}\left(\sqrt{82}+9\right)^{1/8}
\left(\sqrt{\frac{\sqrt{41}+7}{2}}+\sqrt{\frac{\sqrt{41}+5}{2}}\right)^{1/2}\\ && \times\left(\sqrt{\frac{\sqrt{41}+13}{8}}-\sqrt{\frac{\sqrt{41}+5}{8}}\right)^{1/2},\\
G_{130}&=& \frac{1}{2^{1/8}}\left(\sqrt{2}+1\right)^{1/2}\left(\sqrt{10}+3\right)^{1/4}
\left(\sqrt{26}+5\right)^{1/4}\left(5\sqrt{130}+57\right)^{1/8}\\&&
\times\left(\frac{\sqrt{5}-1}{2}\right)^{3/4}\left(\frac{\sqrt{13}-3}{2}\right)^{1/4},\\
G_{142}&=& \frac{1}{2^{1/8}}\left(287\sqrt{71}+1710\sqrt{2}\right)^{1/16}\left(\frac{59+7\sqrt{71}}{\sqrt{2}}\right)^{1/8}
\left(\sqrt{\frac{27\sqrt{2}+40}{4}}+\sqrt{\frac{27\sqrt{2}+36}{4}}\right)^{1/2}\\ && \times\left(\sqrt{\frac{90\sqrt{2}+131}{4}}-\sqrt{\frac{90\sqrt{2}+127}{4}}\right)^{1/4},\\
G_{190}&=& \frac{1}{2^{1/8}}\left(\sqrt{10}-3\right)^{1/4}\left(2\sqrt{5}+\sqrt{19}\right)^{1/4}\left(\sqrt{19}+3\sqrt{2}\right)^{1/4}
\left(37\sqrt{19}+51\sqrt{10}\right)^{1/8}\\&&
\times \left(\frac{3\sqrt{19}+13}{\sqrt{2}}\right)^{1/4}\left(\frac{\sqrt{5}-1}{2}\right)^{3/4}.
\end{eqnarray*}
\end{theorem}
\begin{proof}
Employing class invariant $g_n$ for $ n = 14, 22, 34, 58, 70, 82, 130, 142,$ and $190$ \cite[p. 200-203]{Berndt-notebook-5} in Theorem \ref{thfe3}, we obtain all the above values. Since the proof is analogous to the previous theorem, so we omit the details.
\end{proof}
\begin{theorem}\label{the5}
We have
\begin{eqnarray*}
g_{56}&=&2^{1/8}\left(2\sqrt{2}+\sqrt{7}\right)^{1/16}\left(\frac{3+\sqrt{7}}{\sqrt{2}}\right)^{1/8}
\left(\sqrt{\frac{2\sqrt{2}+3}{2}}+\sqrt{\frac{2\sqrt{2}+1}{2}}\right)^{1/4}\\ && \times\left(\sqrt{\frac{2\sqrt{2}+3}{4}}+\sqrt{\frac{2\sqrt{2}-1}{4}}\right)^{1/4}\\
g_{88}&=& 2^{1/8}\left(\sqrt{2}+1\right)^{1/4}\left(3\sqrt{11}+7\sqrt{2}\right)^{1/8}\left(\frac{\sqrt{11}+3}{\sqrt{2}}\right)^{1/4},\\
g_{136}&=&2^{1/8}\left(\sqrt{2}+1\right)^{1/4}\left(3\sqrt{2}+\sqrt{17}\right)^{1/8}
\left(\sqrt{\frac{\sqrt{17}+5}{4}}+\sqrt{\frac{\sqrt{17}+1}{4}}\right)^{1/2}\\ && \times\left(\sqrt{\frac{3\sqrt{17}+13}{8}}+\sqrt{\frac{3\sqrt{17}+5}{8}}\right)^{1/4},\\
g_{184}&=& 2^{1/8}\left(78\sqrt{2}+23\sqrt{23}\right)^{1/16}\left(\frac{5+\sqrt{23}}{\sqrt{2}}\right)^{1/8}
\left(\sqrt{\frac{3\sqrt{2}+8}{4}}+\sqrt{\frac{3\sqrt{2}+4}{4}}\right)^{1/2}\\ && \times\left(\sqrt{\frac{6\sqrt{2}+11}{4}}+\sqrt{\frac{6\sqrt{2}+7}{4}}\right)^{1/4},\\
g_{232}&=& 2^{1/8}\left(\sqrt{2}+1\right)^{3/4}\left(13\sqrt{58}+99\right)^{1/8}\left(\frac{\sqrt{29}+5}{2}\right)^{1/4},\\
g_{280}&=& 2^{1/8}\left(\sqrt{2}+1\right)^{1/4}\left(2\sqrt{2}+\sqrt{7}\right)^{1/4}\left(3\sqrt{14}+5\sqrt{5}\right)^{1/8} \left(\frac{\sqrt{7}+\sqrt{5}}{\sqrt{2}}\right)^{1/4}\\&&
\times\left(\frac{3+\sqrt{7}}{\sqrt{2}}\right)^{1/4}\left(\frac{\sqrt{5}+1}{2}\right)^{1/2},\\
g_{328}&=& 2^{1/8}\left(\sqrt{2}+1\right)^{1/2}\left(\sqrt{82}+9\right)^{1/8}
\left(\sqrt{\frac{\sqrt{41}+7}{2}}+\sqrt{\frac{\sqrt{41}+5}{2}}\right)^{1/2}\\ && \times\left(\sqrt{\frac{\sqrt{41}+13}{8}}+\sqrt{\frac{\sqrt{41}+5}{8}}\right)^{1/2},\\
g_{520}&=& 2^{1/8}\left(\sqrt{2}+1\right)^{1/2}\left(\sqrt{10}+3\right)^{1/4}
\left(\sqrt{26}+5\right)^{1/4}\left(5\sqrt{130}+57\right)^{1/8}\\&&
\times\left(\frac{\sqrt{5}+1}{2}\right)^{3/4}\left(\frac{\sqrt{13}+3}{2}\right)^{1/4},
\end{eqnarray*}
\begin{eqnarray*}
g_{568}&=& 2^{1/8}\left(287\sqrt{71}+1710\sqrt{2}\right)^{1/16}\left(\frac{59+7\sqrt{71}}{\sqrt{2}}\right)^{1/8}
\left(\sqrt{\frac{27\sqrt{2}+40}{4}}+\sqrt{\frac{27\sqrt{2}+36}{4}}\right)^{1/2}\\ && \times\left(\sqrt{\frac{90\sqrt{2}+131}{4}}+\sqrt{\frac{90\sqrt{2}+127}{4}}\right)^{1/4},\\
g_{760}&=& 2^{1/8}\left(\sqrt{10}+3\right)^{1/4}\left(2\sqrt{5}+\sqrt{19}\right)^{1/4}\left(\sqrt{19}+3\sqrt{2}\right)^{1/4}
\left(37\sqrt{19}+51\sqrt{10}\right)^{1/8}\\&&
\times \left(\frac{3\sqrt{19}+13}{\sqrt{2}}\right)^{1/4}\left(\frac{\sqrt{5}+1}{2}\right)^{3/4}.
\end{eqnarray*}
\end{theorem}
\begin{proof}
Employing pervious theorems in \eqref{identy}, we obtain all the above values.
\end{proof}
\begin{theorem}\label{the6}
We have
\begin{eqnarray*}
\alpha_{14}&=&\left(2\sqrt{2}-\sqrt{7}\right)\left(\frac{3-\sqrt{7}}{\sqrt{2}}\right)^{2}
\left(\sqrt{\frac{2\sqrt{2}+3}{2}}-\sqrt{\frac{2\sqrt{2}+1}{2}}\right)^{4},\\
\alpha_{22}&=& \left(3\sqrt{11}-7\sqrt{2}\right)^{2}\left(\frac{\sqrt{11}-3}{\sqrt{2}}\right)^{4},\\
\alpha_{34}&=&\left(\sqrt{2}-1\right)^{4}\left(3\sqrt{2}-\sqrt{17}\right)^{2}
\left(\sqrt{\frac{\sqrt{17}+5}{4}}-\sqrt{\frac{\sqrt{17}+1}{4}}\right)^{8},\\
\alpha_{46}&=& \left(78\sqrt{2}-23\sqrt{23}\right)\left(\frac{5-\sqrt{23}}{\sqrt{2}}\right)^{2}
\left(\sqrt{\frac{3\sqrt{2}+8}{4}}-\sqrt{\frac{3\sqrt{2}+4}{4}}\right)^{8},\\
\alpha_{58}&=& \left(\sqrt{2}-1\right)^{12}\left(13\sqrt{58}-99\right)^{2},\\
\alpha_{70}&=& \left(2\sqrt{2}-\sqrt{7}\right)^{4}\left(3\sqrt{14}-5\sqrt{5}\right)^{2} \left(\frac{\sqrt{7}-\sqrt{5}}{\sqrt{2}}\right)^{4}
\left(\frac{3-\sqrt{7}}{\sqrt{2}}\right)^{4},\\
\alpha_{82}&=& \left(\sqrt{2}-1\right)^{8}\left(\sqrt{82}-9\right)^{2}
\left(\sqrt{\frac{\sqrt{41}+7}{2}}-\sqrt{\frac{\sqrt{41}+5}{2}}\right)^{8},\\
\alpha_{130}&=& \left(\sqrt{2}-1\right)^{8}\left(\sqrt{10}-3\right)^{4}
\left(\sqrt{26}-5\right)^{4}\left(5\sqrt{130}-57\right)^{2},\\
\alpha_{142}&=& \left(287\sqrt{71}-1710\sqrt{2}\right)\left(\frac{59-7\sqrt{71}}{\sqrt{2}}\right)^{2}
\left(\sqrt{\frac{27\sqrt{2}+40}{4}}-\sqrt{\frac{27\sqrt{2}+36}{4}}\right)^{8},\\
\alpha_{190}&=& \left(2\sqrt{5}-\sqrt{19}\right)^{4}\left(\sqrt{19}-3\sqrt{2}\right)^{4}
\left(37\sqrt{19}-51\sqrt{10}\right)^{2} \left(\frac{3\sqrt{19}-13}{\sqrt{2}}\right)^{4}.
\end{eqnarray*}
\end{theorem}
\begin{proof}
Employing \eqref{eqgn1} in \eqref{e87q1}, we obtain $\alpha_n = \left(G_ng_{4n}\right)^{-8}$. Applying theorems  \ref{tkhe4}-\ref{the5} in the identity, we obtain all the above values.
\end{proof}
\begin{theorem}\label{tbhe5} We have
\begin{eqnarray*}
a_{2, 2} &=& \frac{1}{2^{7/8}}\left(\sqrt{2}+1\right)^{1/2},\\
a_{7, 2} &=& \left(2\sqrt{2}+\sqrt{7}\right)^{1/4}\left(\frac{3+\sqrt{7}}{\sqrt{2}}\right)^{1/2}
\left(\sqrt{\frac{2\sqrt{2}+3}{4}}-\sqrt{\frac{2\sqrt{2}-1}{4}}\right)^{3},\\
a_{11, 2} &=& \left(\sqrt{2}-1\right)^3\left(3\sqrt{11}+7\sqrt{2}\right)^{1/2},\\
a_{15, 2} &=& \left(\sqrt{10}-3\right)\left(\sqrt{6}+\sqrt{5}\right)^{1/2}\left(\frac{\sqrt{5}+\sqrt{3}}{\sqrt{2}}\right)
\left(\frac{\sqrt{5}-1}{2}\right)^3,\\
a_{17, 2} &=&
\left(\sqrt{\frac{\sqrt{17}+5}{4}}+\sqrt{\frac{\sqrt{17}+1}{4}}\right)^{2}
\left(\sqrt{\frac{3\sqrt{17}+13}{8}}-\sqrt{\frac{3\sqrt{17}+5}{8}}\right)^{3},\\
a_{21, 2} &=& \left(\sqrt{2}+1\right)\left(2\sqrt{2}-\sqrt{7}\right)
\left(\frac{\sqrt{3}+1}{\sqrt{2}}\right)^2\left(\frac{\sqrt{7}-\sqrt{3}}{2}\right)^3,\\
a_{23, 2} &=& \left(78\sqrt{2}+23\sqrt{23}\right)^{1/4}\left(\frac{5+\sqrt{23}}{\sqrt{2}}\right)^{1/2}
\left(\sqrt{\frac{6\sqrt{2}+11}{4}}-\sqrt{\frac{6\sqrt{2}+7}{4}}\right)^{3},\\
a_{29, 2} &=& \left(\sqrt{2}+1\right)^3\left(\frac{\sqrt{29}-5}{2}\right)^3,\\
a_{35, 2} &=& \left(\sqrt{2}-1\right)^3\left(3\sqrt{14}+5\sqrt{5}\right)^{1/2}\left(\frac{3+\sqrt{7}}{\sqrt{2}}\right)
\left(\frac{\sqrt{5}-1}{2}\right)^6,\\
a_{39, 2} &=& \left(\sqrt{26}-5\right)\left(\sqrt{13}+2\sqrt{3}\right)\left(3\sqrt{3}+\sqrt{26}\right)^{1/2}
\left(\frac{\sqrt{13}-3}{2}\right)^3,\\
a_{41, 2} &=&
\left(\sqrt{\frac{\sqrt{41}+7}{2}}+\sqrt{\frac{\sqrt{41}+5}{2}}\right)^{2}
\left(\sqrt{\frac{\sqrt{41}+13}{8}}-\sqrt{\frac{\sqrt{41}+5}{8}}\right)^{6},\\
a_{51, 2} &=& \left(\sqrt{2}-1\right)^3\left(\sqrt{3}+\sqrt{2}\right)^2\left(3\sqrt{2}-\sqrt{17}\right)^2
\left(\sqrt{51}+5\sqrt{2}\right)^{1/2},\\
a_{65, 2} &=& \left(\sqrt{10}+3\right)\left(\sqrt{26}+5\right)
\left(\frac{\sqrt{13}-3}{2}\right)^3\left(\frac{\sqrt{5}-1}{2}\right)^9,\\
a_{71, 2} &=& \left(287\sqrt{71}+1710\sqrt{2}\right)^{1/4}\left(\frac{59+7\sqrt{71}}{\sqrt{2}}\right)^{1/2}
\left(\sqrt{\frac{90\sqrt{2}+131}{2}}-\sqrt{\frac{90\sqrt{2}+127}{2}}\right)^{3},
\end{eqnarray*}
\begin{eqnarray*}
a_{95, 2} &=& \left(\sqrt{10}-3\right)^3\left(2\sqrt{5}+\sqrt{19}\right)\left(37\sqrt{19}+51\sqrt{10}\right)^{1/2}
\left(\frac{\sqrt{5}-1}{2}\right)^9.
\end{eqnarray*}
\end{theorem}
\begin{proof}
We set $m = 2, 7, 11, 15, 17, 21, 23, 29, 35, 39, 41, 51, 65, 71,$ and $95$ in Theorem \ref{the3}, and use the corresponding values of $g_{2m}$
from \cite[ Theorem 4.1.2 (i)]{Yi-Thesis-2000}, and \cite[p. 200-203]{Berndt-notebook-5} to complete the proof.
\end{proof}

\end{document}